\documentclass[12pt,twoside]{amsart}
\usepackage{bbm}
\usepackage{verbatim}
\usepackage{graphicx,color}
\usepackage{enumerate}
\usepackage[all]{xy}
\usepackage[hyperindex=true,plainpages=false,colorlinks=false,pdfpagelabels]{hyperref}
\theoremstyle{plain}
\newtheorem{The}{Theorem}%[section]
\newtheorem*{The*}{Theorem}

\newtheorem{Lem}{Lemma}

\newtheorem*{Cor*}{Corollary}

\theoremstyle{definition}
\newtheorem*{Def}{Definition}
\newtheorem{Rem}{Remark}
\newtheorem{Exa}{Example}
\newtheorem*{Rem*}{Remark}

\numberwithin{equation}{section}

\renewcommand{\Im}{\operatorname{Im}}
\renewcommand{\Re}{\operatorname{Re}}

\newcommand{\R}{\mathbb{R}}
\newcommand{\Q}{\mathbb{Q}}
\newcommand{\C}{\mathbb{C}}
\newcommand{\N}{\mathbb{N}}
\newcommand{\Z}{\mathbb{Z}}

\renewcommand{\H}{\mathbb{H}}

\newcommand{\ii}{\mathbbm i}
\newcommand{\jj}{\mathbbm j}
\newcommand{\kk}{\mathbbm k}

\begin{document}

\title{Dirac Tori}

\author{Lynn Heller}

\address{ Institut f\"ur Mathematik\\  Universit{\"a}t T\"ubingen\\ Auf der Morgenstelle
10\\ 72076 T\"ubingen\\ Germany
 }
 \email{lynn-jing.heller@uni-tuebingen.de}

%\subjclass[2010]{Primary 53A05, 53 A 30, 53C42; Secondary 37K15}

%\subjclass{53A10,53C42,53C43,14H60}

\date{\today}

%\thanks{The author is supported by the Sonderforschungsbereich Transregio 71 of the DFG}

\begin{abstract}
We consider conformal immersions $f: T^2\rightarrow \R^3$ with the property that $H^2 f^*g_{\R^3}$ is  a flat metric. These so called Dirac tori have the property that its Willmore energy is uniformly distributed over the surface and can be obtained using spin transformations of the plane by eigenvectors of the standard Dirac operator for a fixed eigenvalue. We classify Dirac tori and determine the conformal classes  realized by them. We want to note that the spinors of Dirac tori satisfies the same system of PDE's as the differential of  Hamiltonian stationary Lagrangian tori in $\R^4$. These were classified in \cite{HR} . 
 \end{abstract}
  
 %\keywords{ Constrained Willmore tori,  spectral curve, CMC tori}
\maketitle

%%%%%%%%%%%%%%%%%%%%%%%%%%%%%%%%%%%%%%%%%%%%%%%%%%%%%%%%%%%%%%%%%%%%%%%%%%%%%%
%           Introduction                                                     %
%%%%%%%%%%%%%%%%%%%%%%%%%%%%%%%%%%%%%%%%%%%%%%%%%%%%%%%%%%%%%%%%%%%%%%%%%%%%%%

\section{Introduction}

Dirac surfaces are defined to be conformal immersions from a (compact) Riemann surface $M$ into $\R^3$ such  that the metric $H^2 f^*g_{\R^3}$ has constant Gau\ss ian curvature, where $H$ is the mean curvature and $f^*g_{\R^3}$ is 
the first fundamental form of the surface. 
Thus Dirac surfaces are obtained by spin transformations of a reference surface - an immersed surface of constant Gau\ss ian and mean curvature in $\R^3$- by eigenvectors of the corresponding Dirac operator, see equation \eqref{Rhoequation}.  For compact and oriented surfaces the sign of the Gau\ss ian curvature is already determined by the topology of the surface and there are three cases to consider. The universal covering of $M$ is either a sphere, a plane or the hyperbolic plane.
In the case of positive  Gau\ss ian curvature we get thus Dirac spheres which are well understood by now, see \cite{BP}. The negative Gau\ss ian curvature case is more difficult to handle since the hyperbolic plane cannot be isometrically immersed into $\R^3$ by a theorem of Hilbert. For vanishing Gau\ss ian curvature the compact Riemann surface $M$ is a torus and the universal covering of $M$ is the plane. Using the plane as reference surface, i.e., the metric  $H^2 f^*g_{\R^3}$ is isometric to a constant multiple of the standard metric $g_{\R^2},$ we get that $H |df|$ is constant for Dirac tori. Moreover, since the Willmore energy is given by $\int_M H^2 |df|^2$, the Willmore energy of Dirac surfaces with flat metric $H^2 f^*g_{\R^3}$ is uniformly distributed over the surface.  Dirac surfaces with a doubly periodic differential can be parametrized using trigonometric functions. In order to obtain a doubly periodic surface we need to investigate closing conditions. It turns out that only conformal classes satisfying certain rationality conditions provide Dirac tori, see theorem \ref{Dirac}.

In this paper we first explain the general setup of spin transformations and how eigenvectors of Dirac operator conformally transforms a reference surface. Then we show that the Dirac surface property can be translated into an eigenvalue problem and derive the corresponding differential equations. Thereafter, we compute the double periodic spin transformations and compute the closing conditions to obtain closed Dirac tori. Further, we determine which conformal classes actually provide Dirac tori. \\

The author would like to thank Ulrich Pinkall for helpful discussions.

\section{Classification}
Let $f: M \rightarrow \R^3$ be a conformal immersion from a Riemann surface $M$ with first fundamental form $f^*g_{\R^3}$ and mean curvature $H$. The immersion is called a {\it Dirac surface} if $H^2 f^*g_{\R^3} $ is a metric with constant Gau\ss ian curvature.
The property of being a Dirac surface is scale invariant. Further, if $M$ is a compact Riemann surface, the topology of the surface determines the sign of the Gau\ss ian curvature of the metric $H^2 f^*g_{\R^3}.$  

We consider the euclidean $3-$space as the space of purely imaginary quaternions $\Im
\H := \{a \in \H | a + \bar a = 0 \}$. In this picture the quaternionic multiplication coincides with the cross product of $\R^3$ and any stretch rotation $R$ of $\R^3$ is given by $$R(x)  = \bar \lambda x \lambda,$$ for an appropriate $\lambda \in \H$ and vise versa. 
The quaternions $\H$ can be turned into a complex vector space by choosing the right multiplication with $\ii$ to be the complex structure. Thus  we can identify $\H = \C \oplus \jj \C$ and every $\lambda \in \H$ can be written as $\lambda = \lambda_1 + \jj \lambda_2$ for two complex numbers $\lambda_1$ and $\lambda_2.$

Let $\lambda: M \rightarrow \H$ be a quaternionic valued function. Then we consider the $1-$form
$$\eta = \bar \lambda df \lambda.$$
On a simply connected domain $\eta$ integrates to a surface $\tilde f$ if and only if $d\eta = 0.$
Then, following \cite{KPP}, we call the map $\tilde f$  a {\it spin transformation} of $f$. A reformulation of the condition
 $d\eta=0$ is stated in the lemma below, which can be found in \cite{CPS, KPP}
\begin{Lem}[\cite{KPP, CPS}]\label{lemma}
Let $f: U \subset M  \rightarrow \R^3$ be a conformal immersion and let $\eta = \bar \lambda df \lambda.$
The $1-$form $\eta$ locally integrates to a conformal immersion $\tilde f$ if and only if there is a real valued function $\rho$ with 
$$D \lambda = \rho \lambda,$$
where $$D \lambda := -\frac{df \wedge d\lambda}{f^*vol}$$ is the Dirac operator of the immersion $f.$ 
\end{Lem}

\begin{proof}
The exterior derivative of $\eta$ is given by
$$d \eta = \bar \lambda df \wedge \lambda + d \bar \lambda \wedge df \lambda.$$
Since $f$ maps into the space of purely imaginary quaternions, i.e., $\bar f = -f,$
we obtain $$\overline{d \bar \lambda \wedge df \lambda} = - \bar \lambda df \wedge d\lambda.$$
Thus $d\eta = 0$ if and only if $ \bar \lambda df \wedge d\lambda$ is purely real valued.

This means that there is a real $2-$form $\rho (f^*vol)$ with
$\bar \lambda df \wedge d\lambda = \rho (f^*vol)$ or equivalently $$-\frac{df \wedge d\lambda}{f^*vol} = \rho \lambda.$$
\end{proof}

A natural question is how the invariants of the surface are transformed. This leads naturally to the notion of the mean curvature half density:
\begin{Def}
Let $f: M \rightarrow \R^3$ be a conformal immersion with mean curvature $H$. Let 
$|df|$ be a $1-$form given by 
$|df| : TM \rightarrow \R, X \mapsto df(X)$.
Then  we call the quantity $H |df|$ the mean curvature half density of $f$.
\end{Def}
\begin{Lem}[\cite{KPP}]
Let $\tilde f$ be a spin transformation of $f.$  With the same notations as in lemma \ref{lemma} we have

\begin{equation}\label{Rhoequation}
\tilde H |d\tilde f| = H |df| + \rho |df|.\end{equation}
\end{Lem}
\begin{Rem}
For a Dirac surface $f: \C \rightarrow \R^3$
we call the $1-$form $H|df|$ constant, if  $H|df(X)|$ is a constant function for a constant vector field $X$ on $\C$ with respect to the standard metric. The mean curvature half density is scale invariant. Further, it  measures the local bending of the surface, since for a torus $\C/\Gamma$ the Willmore energy $\int_{\C/\Gamma} H^2 |df(X)|^2 vol_{\C}$ is the
$L^2-$norm of its mean curvature half density.
 \end{Rem}

\begin{Rem}
Let $f_0$ be a conformally immersed surface in $\R^3$ with constant Gau\ss ian curvature and constant mean curvature. Further, let $g$ be the induced metric of $f_0$. Since a Dirac surface is characterized by the property that $H^2 f^*g_{\R^3}$ has constant Gau\ss ian curvature we have $H^2 f^*g_{\R^3} = c g$ for a constant $c.$ For a conformal map $f$ there is always a local spin transformation with $df = \bar \lambda d f_0 \lambda.$ Thus the corresponding $\rho$ must be constant and we need to solve an eigenvalue problem to obtain Dirac surfaces.
\end{Rem}

We want to specialize in the following to the case where $H^2 f^*g_{\R^3}$ is a flat metric, i.e., where the universal covering of $M$ is a plane. We can use the $\{\jj, \kk\}-$plane as the reference surface. 
The Dirac operator of the plane can be easily computed.
\begin{Lem}[\cite{CPS}]
Let $f = \jj z $ be  the $\{\jj, \kk\}-$plane, where $z$ is the coordinate on $\C$. Then is Dirac operator its given by 
$$D =  \begin{pmatrix} 0 & 2\ii \partial \\ 2\ii\bar \partial & 0\end{pmatrix},$$
where $\partial = \tfrac{\partial}{\partial z}$ and $\partial = \tfrac{\partial}{\partial \bar z}.$

\end{Lem}
\begin{proof}
Let $\lambda : M \rightarrow \H$ be a quaternionic valued function. Then we can write it as $\lambda = \lambda_1 + \jj \lambda_2$ for two complex functions $\lambda_1, \lambda_2$ and
\begin{equation*}
df \wedge d\lambda = \jj dz \wedge (\partial \lambda_1 dz + \bar \partial \lambda_1 d\bar z + \jj \partial \lambda_2 dz +\jj \bar \partial \lambda_2 d \bar z).
\end{equation*}
Since $\jj \lambda_i = \bar \lambda_i \jj$ and $\jj dz = d\bar z \jj$ we get
\begin{equation*}
df \wedge d\lambda = \jj \bar \partial \lambda_1 dz \wedge d \bar z + \partial \lambda_2 dz \wedge d\bar z.
\end{equation*}
The volume form of the plane is given by $  dz \wedge d\bar z  = - 2\ii (\jj z)^*vol$, hence we obtain the claimed form for the Dirac operator.
\end{proof}

We have shown that Dirac surfaces for which $H^2 f^*g_{\R^3}$ is a flat metric are given by the following system of differential equations.
\begin{The}
Let $f : \C \rightarrow \R^3$ be a Dirac surface such that $H^2 f^*g_{\R^3}$ is a flat metric. Then 
$f$ is given by a spin transformation of the plane $f_0 = \jj z$ by a spinor field $\lambda$ solving the eigenvalue problem
$$D \lambda = \mu \lambda, \quad \mu \in \R.$$ 
More precisely, for $\lambda = \lambda_1 + \jj \lambda_2,$ where $\lambda_1, \lambda_2$ are complex valued functions, we obtain the following system of equations:
\begin{equation}\label{PDE}
\begin{split}
2 \ii\partial \lambda_1 &=  \mu \lambda_2 \\
2 \ii\bar \partial \lambda_2 &= \mu \lambda_1.
\end{split} 
\end{equation}
\end{The}

\subsection{Doubly periodic solutions}
In order to obtain compact solutions we specialize to surfaces with doubly periodic differentials. This means that the differential of the surface are defined on the torus $\C / \Gamma$
for a appropriate lattice $\Gamma.$ Complex valued smooth functions on $\C /\Gamma$ are given by a Fourier series
$$\psi = \sum_{\omega \in \Gamma^* }a_\omega e^{\ii<\omega, z>}, \quad a_\omega \in \C,$$

where $\Gamma^* = \{\omega \in \C | <\omega, \gamma> \in 2 \pi \Z, \text{ for } \gamma \in \Gamma \}$ is the dual lattice to $\Gamma.$

For $\lambda = \lambda_1 + \jj \lambda_2$ we have 
\begin{equation}\label{solutions}df = (\jj \lambda_1^2 + \lambda_1\bar \lambda_2)dz + (\jj \lambda_2^2 - \bar\lambda_1\lambda_2)d\bar z,
\end{equation} 
thus the functions $\lambda_i$ need only to be periodic on a double covering of $\Gamma.$  To be more precise we have
\begin{Lem}\label{lambdas}

A solution to equation \eqref{PDE} such that $df = \bar \lambda \jj dz \lambda$ is doubly periodic  is given by 

\begin{equation}
\label{lambdas2}
\begin{split}
\lambda_1 &= \sum_{\omega \in \Gamma'} a_\omega e^{\ii<\omega + \omega_0, z>}\\
\lambda_2 &= -\sum_{\omega \in \Gamma' } \tfrac{1}{\mu} a_\omega (\bar \omega + \bar \omega_0) e^{\ii<\omega + \omega_0, z>},
\end{split} 
\end{equation}
where $\Gamma'  \subset \Gamma^*$ is the set of dual lattice vectors $\omega$ with $|\omega + \omega_0|^2= \mu^2$  and $\omega_0 \in \tfrac{1}{2}\Gamma^*.$
\end{Lem}
\begin{Rem}
The half lattice vector $\omega_0$ corresponds to the spin structure of the spinor $\lambda, $ i.e., the the spin structure of the surface.
\end{Rem}
\begin{Rem}
It is worth to note that the differential of a Hamiltonian stationary Lagrangian torus in $\C^2 \cong \H$, classified in \cite{HR}, satisfies the same system of differential equations as the spinors of Dirac tori, which can be seen as follows. Hamiltonian stationary Lagrangian tori are critical points of the area functional under variations preserving the Lagrangian property. The right normal of the surface $f$, i.e., the $\Im \H$ valued function $N$ given by $df(X) = N df(\ii X)$ is perpendicular to $\ii$ and is thus of the form $N = \jj e^{\ii \beta}. $ Further, the function $\beta$ is harmonic. On a torus  we have $\beta = const <\beta_0, z>,$ for a $\beta_0 \in \Gamma'$, which we can assume without loss of generality to be $\tfrac{1}{2}\ii$. It can be shown that the differential of a Hamiltonian stationary Lagrangian torus is given by $df= e^{\ii \beta/2}(dx +  \jj dy) \lambda, $ for a quaternionic function $\lambda.$ Then the closeness of the $1-$form gives the desired differential equations for $\lambda= \lambda_1 + \jj \lambda_2$. Nevertheless, the closing condition for both surface classes differs.
\end{Rem}

In order to obtain closed surfaces, i.e., Dirac tori, we need to consider closing conditions. A closed  periodic $1-$form on $\C/\Gamma$
$$\eta = \sum_{\omega \in \Gamma^*}e^{\ii<\omega, z>}  (a_\omega dz  + b_\omega  d\bar z) $$
integrates to a periodic function on $\C/\gamma$ if and only if $a_0 = b_0 = 0.$ 

Since $df$ is of the form \eqref{solutions}, the closing condition is  
$$\int_{\C/\Gamma} \lambda_1\bar \lambda_2 = 0 \text{ and } \int_{\C/\Gamma} \lambda_i^2 = 0,$$
which yields the following
\begin{The} \label{The}
Let $df$ be given by \eqref{solutions} with $\lambda_1$ and $ \lambda_2$ as in lemma \ref{lambdas}. Then the Dirac surface $f$ closes to a torus if and only if the following conditions hold
\begin{equation}\label{Closing}
\begin{split}
(1) &\sum_{\omega \in \Gamma'} |a_\omega|^2 (\omega+ \omega_0) = 0\\
(2) &\sum_{\omega \in \Gamma'} a_\omega a _{-\omega- 2 \omega_0} = 0\\
(3) &\sum_{\omega \in \Gamma'}a_\omega a _{-\omega- 2 \omega_0} (\bar \omega + \bar\omega_0)^2 = 0.
\end{split}
\end{equation}
\end{The}

The cardinal number of the set $\Gamma'$, i.e., the lattice vectors of $\Gamma^*$ with  $|\omega + \omega_0|^2= \mu^2$,  is even since for $\omega \in \Gamma'$ we have that $-\omega-2 \omega_0 \in \Gamma'$. If $\#\Gamma' = 2,$ then the only solution to \eqref{PDE} satisfying the closing conditions $(1) - (3)$ is the zero function. If $\Gamma'$ has $4$ elements, i.e., 
$$\Gamma' = \{\omega_1, (-\omega_1-2\omega_0), \omega_2, (-\omega_2-2\omega_0)\},$$ 
then equation $(2), (3)$ in \eqref{Closing} means that the vector $v_0:= (a_{\omega_1}a_{-\omega_1-2\omega_0},a_{\omega_2}a_{-\omega_2-2\omega_0} )$ is perpendicular to the vectors $v_1 := (1,1)$ and $v_2:=((\bar \omega_1 + \bar\omega_0)^2, (\bar \omega_2 + \bar \omega_0)^2)$ in $\C^2$ with respect to the complex linear inner product. Thus $v_0 = (0,0)$ if $v_1$ and $v_2$ are linear independent. But if  $v_2 = c v_1$ we obtain that $\omega_2 = \omega_1$  or $\omega_2 = \omega_1$ and $\Gamma'$ has only $2$ elements. 

 In order to meet all conditions,  we therefore need $\#\Gamma' \geq 6$. 
Let $\omega_1, \omega_2, \omega_3 \in \Gamma' \cap \{z \in \C| \Re(z) \geq 0\}.$ 
We can choose $a_{-\omega_i-2\omega_0} = a_{\omega_i}\ii$ to solve equation $(1)$. For all other $\omega \in \Gamma'$ we choose $a_\omega = 0.$ Then equation $(2)$ and $(3)$ reduces to the fact that the vector $v_0 := (a_{\omega_1}^2, a_{\omega_2}^2, a_{\omega_3}^2)$ is perpendicular to the vectors 
$$(1, 1, 1) \text{ and } ((\bar \omega_1+\bar\omega_0)^2, (\bar \omega_2+\bar\omega_0)^2, (\bar \omega_3+\bar\omega_0)^2)$$ in $\C^3.$ The space of such vectors $v_0$ is at least a complex $1-$dimensional.
\begin{Rem}
Since the Willmore energy is uniformly distributed and $H|df(\tfrac{\partial}{\partial x})| = \mu$, the Willmore energy of a Dirac torus is given by $W(f) = \mu  \cdot vol(\C/\Gamma).$ Thus different solutions of \eqref{lambdas2} to a given eigenvalue $\mu$ and lattice $\Gamma$ have the same Willmore energy.
\end{Rem}
\begin{The}\label{Dirac}
A Dirac torus with a given conformal class exists if and only if there exist a rectangular sub lattice of $\Gamma^*$spanned by two vectors $\omega_1, \omega_2$ with $|\omega_1|^2 = b|\omega_2|^2$ and $b\in \Q.$ \end{The}
\begin{proof}
We first show the statement for $\omega_0 \in \Gamma.$
Let  $\Gamma^*$ be a lattice with a rectangular sub lattice, then we can assume without loss of generality  that this sub lattice is spanned by  the vectors $1, i\tau$ for $\tau \in \R \setminus \{1\}.$ For $\tau = 1$ we have a square lattice and there always exist a non trivial solution.
We want to show that we have at least $6$ vectors with the same length in $\Gamma^*$  if  $\tau^2 \in \Q.$

In this case let $q \in \Z$ with $q \tau^2 = p \in \N.$  Then we can define $m = p-q$ and $n= 2q$ and obtain 
$$(p-q)^2 + 4q^2 \tau^2 = (p-q)^2 + 4pq = (p+q)^2$$
Thus there exist at least $6$ lattice vectors with length $(p+q)^2$, namely $\pm(p-q) \pm i 2q \tau$  and $\pm(p+q)$.

On the other hand if we have $6$ elements of $\Gamma^*$ with the same length,  we can assume that the vectors are given by
$$1, -1, \pm (x_1 + iy_1), \pm (x_2 + iy_2). $$
Since these vectors spann a lattice, we get that there are integers $k, l$ and $m \in \Z$ with 

$$k + l(x_1 + iy_1) = m (x_2 + iy_2)$$
and therefore 
$$ (k + l x_1)^2 + l y_1^2= m^2,$$
where we used that $x_2+ iy_2$ lie on the unit circle.
Thus $x_1 \in \Q$ and there exist $a \in \Z$ with $ax_1 \in \Z$ Moreover, since $x_1^2 + y_1^2  = 1$ we have that $y_1^2 \in \Q,$ and we obtain that $1$ and $iay_1$ is a rectangular sub lattice of the desired form. 

If $\omega_0 \neq 0$ mod $\Gamma, $ we have that $\omega_0 + \Gamma^* \subset \tfrac{1}{2}\Gamma^*.$ Thus a necessary condition for $\# \Gamma' \geq 6$ is that $\tfrac{1}{2}\Gamma^*$ has a rectangular sub lattice of the in the stated form. But then also $\Gamma^*$ has such a sub lattice and there is also a Dirac torus with $\omega_0 = 0$ for this $\Gamma^*$.
\end{proof}

\begin{Exa}
We consider a conformal immersion from $\C/\Gamma$ to $\R^3,$ where the lattice $\Gamma$ is spanned by 
$2 \pi ,$ and $2\pi i$. Then the dual lattice $\Gamma^*$ is spanned by $1$ and $i$. We choose $\omega_0 = 0.$ The smallest possible eigenvalue to obtain a torus is $\mu = \sqrt{5}$ and $\Gamma'$ consists of the vectors
$$2+i, 2-i, 1+ 2i, 1-2i, -2-i, -2+i, -1-2i, -1+2i.$$ 
Then the following choice of coefficients leads to a torus, see figure \ref{figure}. 
\begin{equation}
\begin{split}
&a_{1+2i} =  0, \quad a_{-1-2i} = 0\\
&a_{-2-i} = i a_{2+i}, \quad  a_{-2+i} = i a_{2-i}, \quad  a_{-1+2i} = ia_{1-2i}\\
&a_{2+i} = \sqrt{\tfrac{3}{2}} (1+i), \quad a_{2-i} = \tfrac{1}{\sqrt 2 }(1-3i), \quad a_{1-2i}  = 2.
\end{split}
\end{equation}

The Willmore energy of the surface is given by $\mu \cdot vol(\C/\Gamma) = 4\sqrt5 \pi^2.$
\begin{figure}[htbp]\label{figure}
  \centering
  \includegraphics[scale=0.31]{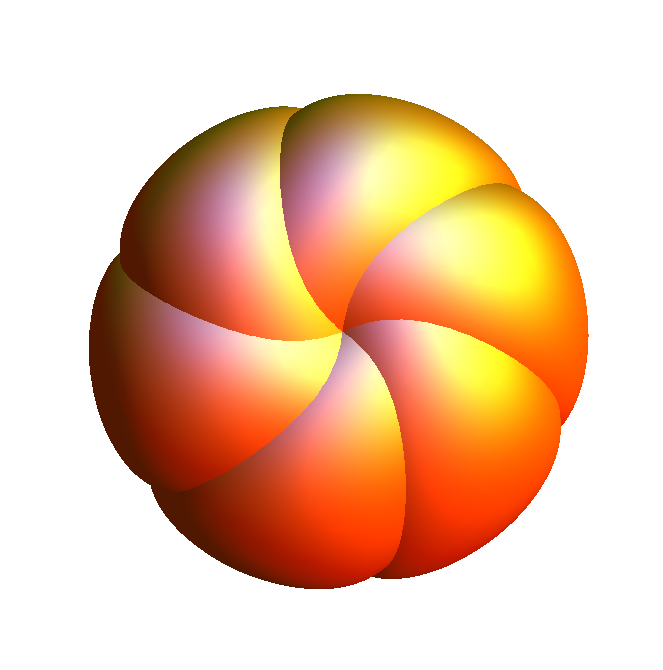}
  \caption{Dirac torus in $\R^3$ with conformal type given by $2\pi$ and $2\pi i.$ }
\end{figure}

\end{Exa}

%%%%%%%%%%%%%%%%%%%%%%%%%%%%%%%%%%%%%%%%%%%%%%%%%%%%%%%%%%%%%%%%%%%%%%%%%%%%%%
%           Bibliography                                                     %
%%%%%%%%%%%%%%%%%%%%%%%%%%%%%%%%%%%%%%%%%%%%%%%%%%%%%%%%%%%%%%%%%%%%%%%%%%%%%%

\end{document}